\documentclass[11pt]{article}

\usepackage{mathtools}
\usepackage{mathrsfs}
\usepackage{enumerate}
\usepackage{amsmath}
\usepackage{amsthm}
\usepackage{amssymb}
\usepackage{amsfonts}
\usepackage{mathabx}
\usepackage{setspace}
\usepackage{color}
\usepackage{comment}
\usepackage{tikz}
\usepackage{tikz-cd}
\usepackage{amscd}
\usepackage{scalerel}
\usepackage{hyperref}

\theoremstyle{plain}
\newtheorem{Theorem}{Theorem}[section]
\newtheorem{Proposition}[Theorem]{Proposition}
\newtheorem{Lemma}[Theorem]{Lemma}
\newtheorem{Corollary}[Theorem]{Corollary}
\newtheorem{Properties}[Theorem]{Properties}

\newtheorem{Conjecture}[Theorem]{Conjecture}

\theoremstyle{definition}
\newtheorem{Definition}[Theorem]{Definition}
\newtheorem{Example}[Theorem]{Example}
\newtheorem{Examples}[Theorem]{Examples}
\newtheorem{Remark}[Theorem]{Remark}

\newcommand{\cat}[1]{\text{\rm cat}\left(#1 \right)}
\newcommand{\TC}[1]{\text{\rm TC}\left(#1 \right)}

\newcommand{\secat}[1]{\text{\rm secat}\left(#1 \right)}
\newcommand{\secatg}[1]{\text{\rm secat}_g\left(#1 \right)}
\newcommand{\relTC}[2]{\text{\rm TC}_{#1} \left(#2 \right)}
\newcommand{\pullTC}[1]{\text{\rm TC} \left( #1 \right)}
\newcommand{\halfTC}[1]{\text{\rm TC}^{\frac{1}{2}} \left( #1 \right)}
\newcommand{\robotTC}[1]{\text{\rm TC}^{ra} \left( #1 \right)}
\newcommand{\pairTC}[2]{\text{\rm TC} \left( #1, #2 \right)}
\newcommand{\MTC}[1]{\text{\rm TC}^M \left( #1 \right)}

\newcommand{\cuplength}[1]{\text{\rm cup} \left( #1 \right)}

\newcommand{\subsize}[1]{\scaleto{#1}{4pt}}

\newcommand{\cd}[1]{\text{\rm cd} \left( #1 \right)}

\title{On the Topological Complexity of  Maps}
\author{Jamie Scott}

\begin{document}
\maketitle

\begin{abstract}
We define and develop a homotopy invariant notion for the topological complexity of a map $f:X \to Y$, denoted $\pullTC{f}$, that interacts with $\TC{X}$ and $\TC{Y}$ in the same way $\cat{f}$ interacts with $\cat{X}$ and $\cat{Y}$. Furthermore, $\pullTC{f}$ and $\cat{f}$ satisfy the same inequalities as $\TC{X}$ and $\cat{X}$. We  compare it to other invariants defined in the papers \cite{Murillo_Wu, Pavesic_1,Pavesic_2,Pavesic_3, Short}. We apply $\pullTC{f}$ to studying group homomorphisms $f:H\to G$.
\end{abstract}

\section{Introduction}

Let $X$ denote the configuration space of a mechanical system, then a navigation algorithm of the system is a map $\sigma:X\times X\to X^I$, $I=[0,1]$, to the path space such that $\sigma(x_0,x_1)(0)=x_0$ and $\sigma(x_0,x_1)(1)=x_1$. Clearly, if there exists a continuous navigation algorithm, then $X$ is contractible. Since configuration spaces are rarely contractible, one has to do partitioning of $X\times X$ into  pieces such that on each piece there is a continuous navigation. M. Farber called the minimal number of such pieces \cite{Farber_2} (also see \cite{Farber_4}) the {\em topological complexity} of $X$
and denoted it by $\TC{X}$. Formally, $\TC{X}$ is the minimal number $k$ such that
$X\times X$ can be covered by open sets $U_0,\dots, U_k$ such that the path space fibration $p:X^I \to X \times X$ admits a section on each $U_i$.
Here we gave the definition of the reduced topological complexity which differs by one from the original.
This invariant was intensively studied in the last 20 years.

When the mechanical system is a robotic arm the $\TC{X}$ input on the navigation problem is not quite satisfactory. In the case of robotic arm one has a map $f:X\to Y$ of the configuration space $X$ to the work space $Y$. In that case in his lectures at the workshop on Applied Algebraic Topology in Castro Urdiales (Spain, 2014) A. Dranishnikov suggested to study the topological complexity of maps $\TC{f}$ and defined it as the minimal number $k$ such that $X\times Y$ can be covered by open sets $U_0,\dots, U_k$ such that the over each $U_i$ there is a section to the map $q:X^I\to X\times Y$ defined as
$q(\phi)=(\phi(0),f\phi(1))$ (see~\cite{Pavesic_1}). The drawback of this definition is that it can be applied  only to maps $f:X\to Y$ that have the path lifting property. All path fibrations and left divisor of path fibrations have this property, but
it is unclear if the forward kinematic maps for general robotic arms are left divisors of fibrations. In~\cite{Pavesic_2, Pavesic_3} P. Pavesic modified this definition to cover general maps, but still his definition  lacks one important feature: It is not a homotopy invariant.

Topological complexity is a homotopy invariant similar to LS-category, in that both can be characterized as the sectional category of a particular fibration, i.e., the topological complexity of a path-connected space $X$ is the sectional category of the path space fibration $X^I \to X \times X$, and the LS-category of $X$, denoted $\cat{X}$, is the sectional category of the based path space fibration $P_0X \to X$. Both of these invariants are notoriously difficult to compute, with lower bounds usually being the more difficult problem.

In the case of LS-category, the notion was extended to maps $f:X \to Y$ satisfying the inequality $\cat{f} \leq \cat{X}, \cat{Y}$. It proved to be helpful in computing the LS-category of a space \cite{Stanley} and it is homotopy invariant. 
The goal of this paper is to define and study analogous an invariant $\TC{f}$
for topological complexity. There was no  invariant occupying this theoretical niche. Only very recently when this work was in progress
A. Murillo and J. Wu defined their own notion of $\TC{f}$ in \cite{Murillo_Wu}, which I will later prove is equivalent to the notion presented here. Notably, with the perspective Murillo and Wu took, they were unable to prove the inequality $\TC{f} \leq \TC{Y}$.

Since cat and TC are homotopy invariants, they produce numerical invariant of discrete groups. In the case of the category by the theorem of Eilenberg and Ganea $\cat{G}$ is just the cohomological dimension of a group $G$. An algebraic description of $\TC{G}$ of a group $G$
is one of the most important problems in the area~\cite{Farber_Mescher}. In this paper I applied both
invariants cat and TC to group homomorphisms. Even in the case of cat, 
despite on easy answer for groups, there is no easy answer for homomorphisms.
In Section 6, we present a conjecture on that (Conjecture \ref{LS Conjecture}) which 
grew out of discussion of the problem with A. Dranishnikov. The conjecture is proved in the paper for free groups and free abelian groups. In the case of topological complexity I don't have any conjecture
of what $\TC{f}$ is for a group homomorphism $f:H\to G$, since there is no corresponding theorem for the topological complexity of groups, $\TC{G}$. Again, when both $G$ and $H$ are free or free abelian we give the answer in the paper.

The paper is organized as follows.
In section 3, we give a definition for $\TC{f}$ that takes inspiration from the fact that $\cat{f}$ is equal to the sectional category of the pullback fibration $f^\ast \pi^Y$, where $\pi^Y:P_0Y \to Y$ is the based path space fibration (see Theorem 19$'$ of \cite{Schwarz}). Thus, we can define $\TC{f}$ to be the sectional category of the pullback of the path space fibration $Y^I \to Y \times Y$ by $f \times f$. As this definition is a bit theoretical in nature, we take inspiration from the various formulations of $\TC{X}$ to define $f$-motion planners and instead start section 3 with this version of the definition. Here we prove various properties for $\TC{f}$, including homotopy invariance, inequalities mirroring those for $\cat{f}$, and inequalities that extend the interaction between $\cat{X}$ and $\TC{X}$ to maps.

In section 4, we consider instead the fibration obtained as the pullback of the path space fibration by $Id_Y \times f$. In \cite{Pavesic_3}, Pave\u{s}i\'c proved that the sectional category of this fibration is equal to his complexity, denoted here as $\robotTC{f}$, if $f$ is a fibration. Here we study this sectional category without that assumption as an intermediary step to get to the pullback by $f \times f$. We call it the mixed topological complexity of $f$, denoted $\halfTC{f}$. Section 4 is structured similarly to section 3, but at the end we prove the inequalities $\pullTC{f} \leq \halfTC{f} \leq \robotTC{f}$.

The use of $f$-motion planners as a means of defining $\pullTC{f}$ inspires definitions for symmetric and monoidal TC of maps, so section 5 of this paper proves some basic properties for the latter and leaves the former for future work.

The final section of this paper uses the homotopy invariance of $\pullTC{f}$ to define the topological complexity of group homomorphisms, as discussed above.

\section{Preliminaries}

\subsection{Sectional Category}

\begin{Definition}
Let $p:E \to B$ be a fibration. Then the {\it sectional category} of $p$, denoted $\secat{p}$, is the smallest integer $k$ such that $B$ can be covered by $k+1$ open sets $U_0,...,U_k$ that each admit a partial section $s_i:U_i \to E$ of $p$. If no such integer $k$ exists, then we set $\secat{p} = \infty$.
\end{Definition}

\begin{Proposition} \label{genus invariance}
Let $p:E \to B$ and $p':E' \to B'$ be fibrations. If there is a homotopy commutative diagram
\[\begin{tikzcd}
E \arrow[r, "\overline{h}"] \arrow[d, swap, "p"]
& E' \arrow[d, "p'"]
\\
B \arrow[r, swap, "h"]
& B'
\end{tikzcd}\]
such that $h$ is a homotopy domination, then $\secat{p'} \leq \secat{p}$.
\end{Proposition}

\begin{proof}
Let $g$ be the right homotopy inverse of $h$, and let $\secat{p}=k$ so that there is an open cover $U_0,...,U_k$ of $B$ and partial sections $s_i:U_i \to E$ of $p$. Let $V_i = g^{-1}(U_i)$, which forms an open cover of $B'$. Now define $s'_i:V_i \to E'$ by $s'_i = \overline{h} s_i g|_{V_i}$. Then
\[ p' s'_i
= p' \overline{h} s_i g|_{V_i}
\simeq h p s_i g|_{V_i}
= h (Id_B|_{U_i}) g|_{V_i}
= h g|_{V_i}\]
but $hg$ is homotopic to $Id_{B'}$ so that $p's'_i$ is homotopic to the inclusion map $V_i \to B'$; hence, $s'_i$ is a homotopy section of $p'$. Therefore, $\secat{p'} \leq \secat{p}$.
\end{proof}

The following is proven in Proposition 7 of \cite{Schwarz}:

\begin{Proposition} \label{sectional pullback}
Let $p:E \to B$ be a fibration, let $f:B' \to B$ be any map, and consider the pullback diagram
\[\begin{tikzcd}
f^\ast E \arrow[r, "\overline{f}"] \arrow[d, swap, "f^\ast p"]
& E \arrow[d, "p"]
\\
B \arrow[r, swap, "f"]
& B'
\end{tikzcd}\]
Then $\secat{f^\ast p} \leq \secat{p}$.
\end{Proposition}

\begin{proof}
Suppose $\secat{p}=k$ and let $U_0,...,U_k$ be an open cover of $B$ with sections $s_i:U_i \to E$ of $p$. Let $V_i=f^{-1}(U_i)$. Using the pullback diagram, there are maps $s_i':V_i \to f^\ast E$ such that the diagrams
\[\begin{tikzcd} [column sep = 3em, row sep = 3 em]
V_i \arrow[drr, shift left, "s_if"] \arrow[dr, "s_i'"] \arrow[ddr, hook, shift right]
\\
& f^\ast E \arrow[r, swap, "\overline{f}"] \arrow[d, "f^\ast p"]
& E \arrow[d, "p"]
\\
& B \arrow[r, swap, "f"]
& B'
\end{tikzcd}\]
commute. Therefore, each $s_i'$ is a section of $f^\ast p$ on $V_i$ so that $\secat{f^\ast p} \leq \secat{p}$.
\end{proof}

The following product formula is Proposition 22 of \cite{Schwarz}, but a more understandable proof is give by Theorem 11 of \cite{Farber_2} for the case of topological complexity. The conditions here are weaker than those of \cite{Farber_2} and \cite{Schwarz}, but the only necessary condition for the proof is that of normality as it ensures the existence of partitions of unity subordinate to finite open covers.

\begin{Proposition} \label{secat product}
Let $p:E \to B$ and $p':E' \to B'$ be fibrations such that $B$ and $B'$ are normal. Then $\secat{p \times p'} \leq \secat{p} + \secat{p'}$.
\end{Proposition}

\subsection{The Fiberwise Join}

For $1 \leq i \leq n$, let $p_i:E_i \to B$ be a fibration with fiber $F_i$. The fiberwise join of the total spaces $E_i$ over the base $B$ is defined to be the subspace of the usual join $E_1 \ast ... \ast E_n$ given by set
\[ E_1 \ast_B ... \ast_B E_n
= \left\{ t_1e_1 + ... + t_ne_n \in E_1 \ast ... \ast E_n \mid t_i,t_j \neq 0 \implies p_i(e_i) = p_j(e_j) \right\}.\]
Then the fiberwise join of the fibrations $p_i$ is the map
\[ p_1 \ast_B ... \ast_B p_n: E_1 \ast_B ... \ast_B E_n \to B\]
given by
\[ t_1e_1 + ... + t_ne_n \mapsto p_i(e_i)\]
for any $i$ such that $t_i \neq 0$. This map is a well defined fibration by the definition of the space $E_1 \ast_B ... \ast_B E_n$.
This operation is called the fiberwise join as the fiber of $p_1 \ast_B ... \ast_B p_n$ is the join of the fibers $F_1 \ast ... \ast F_n$.

If $p_i=p:E \to B$ for all $i$, then we denote the fibration $p_1 \ast_B ... \ast_B p_n$ by $\ast_B^n p$ and we denote its total space $E_1 \ast_B ... \ast_B E_n$ by $\ast_B^n E$.

The following theorem encapsulates the Ganea-Schwarz approach of using fiberwise joins to calculate sectional category:

\begin{Theorem} [\cite{Schwarz}] \label{Ganea-Schwarz Approach}
Let $p:E \to B$ be a fibration with $B$ normal. Then $\secat{p} \leq n$ if and only if $\ast_B^{n+1} p$ admits a section.
\end{Theorem}

\subsection{LS-Category}

\begin{Definition}
The (reduced) Lusternik-Schnirelmann category of a space $X$, denoted $\cat{X}$, is defined to be the smallest integer $n \geq -1$ such that $X$ admits an open cover by $n+1$ sets $U_0,...,U_n$ that are each contractible in $X$.
\end{Definition}

If $X$ is nonempty and path-connected, then the sectional category of the based path space fibration $\pi:P_0(X) \to X$ is exactly $\cat{X}$ so that LS-category can be thought of as a special case of sectional category. Moreover, we call $G_n(X) = \ast^{n+1}_X P_0(X)$ the $n$-th Ganea space of $X$ and $p^X_n = \ast^{n+1}_X \pi$ the $n$-th Ganea fibration of $X$.

\begin{Properties}
Let $X$ and $Y$ be normal and path-connected spaces.
\begin{enumerate}
    \item[$\bullet$] If $X$ is $(r-1)$-connected, then $\cat{X} \leq \frac{\dim X}{r}$.
    \item[$\bullet$] $\cat{X \times Y} \leq \cat{X} + \cat{Y}$.
    \item[$\bullet$] If $X$ is an ANR, then $\cuplength{\widetilde{H}^\ast(X)} \leq \cat{X}$.
    \item[$\bullet$] $\cat{X} \leq n$ if and only if $p^X_n$ admits a section.
\end{enumerate}
\end{Properties}

\begin{Definition}
The (reduced) Lusternik-Schnirelmann category of a map $f:X \to Y$, denoted $\cat{f}$, is defined to be the smallest integer $n \geq -1$ such that $X$ admits an open cover by $n+1$ sets $U_0,...,U_n$ such that each restriction $f|_{U_i}$ is nullhomotopic.
\end{Definition}

\begin{Properties}
Let $f:X \to Y$ and $g:Z \to W$ be maps between normal and path-connected spaces.
\begin{enumerate}
    \item[$\bullet$] $\cat{f} \leq \cat{X}, \cat{Y}$.
    \item[$\bullet$] $\cat{f \times g} \leq \cat{f} + \cat{g}$.
    \item[$\bullet$] $\cuplength{\ker f^\ast} \leq \cat{f}$.
    \item[$\bullet$] $\cat{f} = \secat{f^\ast p_0^Y}$.
    \item[$\bullet$] $\cat{f} \leq n$ if and only if there is a lift of $f$ with respect to $p^Y_n$.
\end{enumerate}
\end{Properties}

\subsection{Topological Complexity}

\begin{Definition}
Let $X$ be a topological space.
\begin{enumerate}
    \item A motion planner on a subset $Z \subset X \times X$ is a map $s:Z \to X^I$ such that $s(x_0,x_1)(0)=x_0$ and $s(x_0,x_1)(1)=x_1$.
    \item A motion planning algorithm is a cover of $X \times X$ by sets $Z_0,...,Z_k$ such that on each $Z_i$ there is some motion planner $s_i:Z_i \to X^I$.
    \item The topological complexity of a space $X$, denoted $\TC{X}$, is the least $k$ such that $X \times X$ can be covered by $k+1$ open subsets $U_0,...,U_k$ on which there are motion planners. If no such $k$ exists, we define $\TC{X} = \infty$.
    \item Given a subset $A \subset X \times X$, the relative topological complexity of $A$ in $X$, denoted $\relTC{X}{A}$, is the least $k$ such that $A$ can be covered by $k+1$ sets $U_0,...U_k$ open in $A$ on which there are motion planners $U_i \to X^I$ for $X$. As before, if no such $k$ exists, we define $\relTC{X}{A} = \infty$.
\end{enumerate}
\end{Definition}

If $X$ is nonempty and path-connected, then the sectional category of the path space fibration $\Delta^X_0:P(X) \to X \times X$ is exactly $\TC{X}$ so that topological complexity can also be thought of as a special case of sectional category. We use the notation $\Delta^X_0$ for the path space fibration since it is a fibration replacement of the diagonal map $\Delta:X \to X \times X$. Moreover, we denote the corresponding fiberwise joins as $\Delta_n(X) = \ast^{n+1}_{X \times X} P(X)$ and $\Delta^X_n = \ast^{n+1}_{X \times X} \Delta^X$.

\begin{Properties}
Let $X$ and $Y$ be normal and path-connected spaces.
\begin{enumerate}
    \item[$\bullet$] $\cat{X} \leq \TC{X} \leq \cat{X \times X}$.
    \item[$\bullet$] $\TC{X \times Y} \leq \TC{X} + \TC{Y}$.
    \item[$\bullet$] If $X$ is an ANR, then $\cuplength{\ker \Delta^\ast} \leq \TC{X}$.
    \item[$\bullet$] $\TC{X} \leq n$ if and only if $\Delta_n^X$ admits a section.
    \item[$\bullet$] $\relTC{X}{X \times X} = \TC{X}$.
    \item[$\bullet$] If $A \subset B \subset X \times X$, then $\relTC{X}{A} \leq \relTC{X}{B}$.
\end{enumerate}
\end{Properties}

\section{The Pullback TC of a Map}

\subsection{Definitions}

\begin{Definition}
Let $f:X \to Y$ be a map.
\begin{enumerate}
    \item An {\it $f$-motion planner} on a subset $Z \subset X \times X$ is a map $f_{\subsize{Z}}:Z \to Y^I$ such that $f_{\subsize{Z}}(x_0,x_1)(0)=f(x_0)$ and $f_{\subsize{Z}}(x_0,x_1)(1)=f(x_1)$.
    \item An {\it $f$-motion planning algorithm} is a cover of $X \times X$ by sets $Z_0,...,Z_k$ such that on each $Z_i$ there is some $f$-motion planner $f_i:Z_i \to Y^I$.
    \item The {\it (pullback) topological complexity} of $f$, denoted $\pullTC{f}$, is the least $k$ such that $X \times X$ can be covered by $k+1$ open subsets $U_0,...,U_k$ on which there are $f$-motion planners. If no such $k$ exists, we define $\pullTC{f} = \infty$.
\end{enumerate}
\end{Definition}

\begin{Remark}
If $f$ doesn't map $X$ into a single path component of $Y$, then we always have $\pullTC{f} = \infty$ by this definition, so to avoid this one would have to add up the TC for each path component $f$ maps into. For the sake of this work, all spaces are assumed to be path-connected.
\end{Remark}

\begin{Examples}
{\color{white} LATEX SPACING ISSUE!!!}
\begin{enumerate}
    \item $\pullTC{f} = -1$ if and only if $f$ is the empty map.
    \item If $i:A \to X$ is an inclusion map, then $\pullTC{i} = \relTC{X}{A \times A}$.
    \item $\pullTC{Id_X} = \TC{X}$ for any space $X$.
\end{enumerate}
\end{Examples}

The following theorem shows that the notion of an $f$-motion planner can be substituted for sections of a fibration or a deformation into the diagonal as is the case with the topological complexity of a space. Additionally, the last item in the following equivalence shows an equivalence with the notion of TC studied in \cite{Murillo_Wu}.

\begin{Theorem} \label{pullback definitions}
Let $f:X \to Y$ be a map and let $Z \subset X \times X$. The following are equivalent:
\begin{enumerate}
    \item there is a partial section $s:Z \to (f \times f)^\ast Y^I$ of the pullback fibration $(f \times f)^\ast \Delta_0^Y$;
    \item there is an $f$-motion planner $f_{\subsize{Z}}:Z \to Y^I$;
    \item $(f \times f)|_Z$ can be deformed into $\Delta Y$
    \item there is a map $h:Z \to X$ such that $(f \times f) \Delta h$ is homotopic to $(f \times f)|_Z$.
\end{enumerate}
\end{Theorem}

\begin{proof}
($\bf 1 \implies 2$) Let $s:Z \to (f \times f)^\ast Y^I$ be a partial section of the pullback fibration $(f \times f)^\ast \Delta_0^Y$ and consider the pullback diagram of the maps $f \times f$ and $\Delta_0^Y$
\[ \begin{tikzcd}[column sep = 4em, row sep = 4em]
(f \times f)^\ast Y^I \arrow[r, "\overline{f \times f}"] \arrow[d, swap, "(f \times f)^\ast \Delta_0^Y"]
& Y^I \arrow[d, "\Delta_0^Y"]
\\
X \times X \arrow[r, swap, "f \times f"]
& Y \times Y
\end{tikzcd}\]
where the $(f \times f)^\ast Y^I$ is given by
\[ (f \times f)^\ast Y^I = \{ (x_0,x_1,\alpha) \in X \times X \times Y^I \mid \alpha(0)=f(x_0) \text{ and } \alpha(1)=f(x_1) \}\]
and $(f \times f)^\ast \Delta_0^Y$ and $\overline{f \times f}$ are the obvious projection maps. Define the $f$-motion planner $f_{\subsize{Z}}:Z \to Y^I$ by $f_{\subsize{Z}} = (\overline{f \times f}) s$. Then it follows that
\[ \Delta_0^Y f_{\subsize{Z}}
= \Delta_0^Y (\overline{f \times f}) s
= (f \times f) \left( (f \times f)^\ast \Delta_0^Y \right) s
= (f \times f) Id_{Z}
= (f \times f)|_{Z}\]
so that applying both sides to some $(x_0,x_1) \in Z$ gives
\[ \left( f_{\subsize{Z}}(x_0,x_1)(0), f_{\subsize{Z}}(x_0,x_1)(1) \right)
= (\Delta_0^Y f_{\subsize{Z}})(x_0,x_1)
= (f \times f)|_{Z} (x_0,x_1)
= \left( f(x_0),f(x_1) \right).\]
Therefore, $f_{\subsize{Z}}$ is an $f$-motion planner on $Z$.

($\bf 2 \implies 1$) Let $f_{\subsize{Z}}:Z \to Y^I$ be an $f$-motion planner. Since $(f \times f)^\ast 
Y^I$ is the pullback of $f \times f$ and $\Delta_0^Y$, there is some $s:Z \to (f \times f)^\ast 
Y^I$ such that the diagram
\[ \begin{tikzcd}[column sep = 3em, row sep = 3em]
Z \arrow[ddr, swap, "i"] \arrow[drr, "f_{\scaleto{Z}{3 pt}}"] \arrow[dr, "s"]
\\
& (f \times f)^\ast Y^I \arrow[r, swap, "\overline{f \times f}"] \arrow[d, "(f \times f)^\ast \Delta_0^Y"]
& Y^I \arrow[d, "\Delta_0^Y"]
\\
&X \times X \arrow[r, swap, "f \times f"]
& Y \times Y
\end{tikzcd}\]
commutes where $i:Z \to X \times X$ is the inclusion map. Therefore, $\left( (f \times f)^\ast \Delta_0^Y \right) s = i$ so that $s$ is a partial section of $(f \times f)^\ast 
\Delta_0^Y$ on $Z$.

($\bf 2 \implies 3,4$) Let $f_{\subsize{Z}}:Z \to Y^I$ be an $f$-motion planner. Now define a homotopy $H:X \times X \times I \to Y \times Y$ by
\[ H(x_0,x_1,t) = \left( f_{\subsize{Z}}(x_0,x_1)(t),f(x_1) \right),\]
which is a deformation of $(f \times f)|_Z$ into $\Delta Y$, showing $(3)$. Also observe that
\[ H(x_0,x_1,1)
= (f(x_1),f(x_1))
= (f \times f)(x_1,x_1)
= (f \times f)\Delta(x_1)
= (f \times f) \Delta (\text{proj}_1)|_Z(x_0,x_1)\]
where $\text{proj}_1:X \times X \to X$ is the projection map into the second coordinate. This proves $(4)$ by setting $h= (\text{proj}_1)|_Z$.

($\bf 3 \implies 2$) Now let $H:Z \times I \to Y \times Y$ be a deformation of $(f \times f)|_Z$ into $\Delta Y$. Define a map $f_{\subsize{Z}}:Z \to Y^I$ by
\[ f_{\subsize{Z}}(x_0,x_1)(t) =
\begin{cases}
\text{proj}_0 \left(H(x_0,x_1,2t) \right) & \text{if } 0 \leq t \leq \frac{1}{2} \\
\text{proj}_1 \left(H(x_0,x_1,2 - 2t) \right) & \text{if } \frac{1}{2} \leq t \leq 1
\end{cases}\]
where $\text{proj}_0,\text{proj}_1:Y \times Y \to Y$ are the projection maps into the corresponding coordinates. Then $f_{\subsize{Z}}$ is well-defined and continuous since $H(x_0,x_1,t) \in \Delta Y$ for all $(x_0,x_1) \in Z$. Moreover,
\[ f_{\subsize{Z}}(x_0,x_1)(0)
= \text{proj}_0 \left(H(x_0,x_1,0) \right)
= \text{proj}_0 \left(f(x_0),f(x_1)) \right)
= f(x_0)\]
and
\[ f_{\subsize{Z}}(x_0,x_1)(1)
= \text{proj}_1 \left(H(x_0,x_1,0) \right)
= \text{proj}_1 \left(f(x_0),f(x_1)) \right)
= f(x_1)\]
so that $f_{\subsize{Z}}$ is an $f$-motion planner on $Z$.

($\bf 4 \implies 3$) Note that the image of $(f \times f) \Delta h$ will always be contained in $\Delta Y$.
\end{proof}

This theorem implies that $\pullTC{f}$ is the minimum $k$ such that there is an open cover $U_0,...,U_k$ on each of which at least one of the above equivalent properties hold. Moreover, if $X$ and $Y$ are ANRs, then any cover will do as we will see in the following theorem.

\begin{Definition}
Let $p:E \to B$ be a fibration. Then the {\it generalized sectional category} of $p$, denoted $\secatg{p}$, is the smallest integer $k$ such that $B$ can be covered by $k+1$ sets $Z_0,...,Z_k$ that each admit a partial section $s_i:Z_i \to E$ of $p$. If no such integer $k$ exists, then we set $\secatg{p} = \infty$.
\end{Definition}

\begin{Theorem} \label{pullback cover ANR}
If $f:X \to Y$ is a map between ANRs, then $\pullTC{f}=\secatg{(f \times f)^\ast \Delta^Y_0}$.
\end{Theorem}

\begin{proof}
Since $X$ and $Y$ are ANRs, it follows that $X \times X$, $Y \times Y$, and $Y^I$ are all ANRs. Since $\Delta^Y_0$ is a fibration, it follows that $(f \times f)^\ast Y^I$ is an ANR (see Theorem 2.2 of \cite{Miyata}). Thus, $(f \times f)^\ast \Delta^Y_0$ is a fibration between ANRs so that $\secat{(f \times f)^\ast \Delta^Y_0} = \secatg{(f \times f)^\ast \Delta^Y_0}$ (see \cite{Garcia-Calcines} and \cite{Srinivasan}). Thus, $\pullTC{f}=\secatg{(f\times f)^\ast \Delta^Y_0}$ by Theorem \ref{pullback definitions}.
\end{proof}

\begin{Corollary} \label{pullback TC lifting}
Let $f:X \to Y$ be any map with $X \times X$ normal. Then $\pullTC{f} \leq k$ if and only if there is a lift of $f \times f$ with respect to $\Delta^Y_{k+1}$.
\end{Corollary}

\begin{proof}
Consider the pullback diagram
\[ \begin{tikzcd}[column sep = 4em, row sep = 4em]
(f \times f)^\ast Y^I \arrow[r, "\overline{f \times f}"] \arrow[d, swap, "(f \times f)^\ast \Delta_0^Y"]
& Y^I \arrow[d, "\Delta_0^Y"]
\\
X \times X \arrow[r, swap, "f \times f"]
& Y \times Y
\end{tikzcd}\]
as usual. Note that taking the self fiberwise join of a fibration commutes with taking pullbacks, so we get the pullback diagram
\[ \begin{tikzcd}[column sep = 4em, row sep = 4em]
(f \times f)^\ast \Delta_k(Y)\arrow[r] \arrow[d, swap, "(f \times f)^\ast \Delta_k^Y"]
& \Delta_k(Y) \arrow[d, "\Delta_k^Y"]
\\
X \times X \arrow[r, swap, "f \times f"]
& Y \times Y
\end{tikzcd}\]
where $\ast_{\scaleto{X \times X}{5pt}}^{k+1}(f \times f)^\ast Y^I = (f \times f)^\ast \Delta_k(Y)$ and $\ast_{\scaleto{X \times X}{5pt}}^{k+1} (f \times f)^\ast \Delta_0^Y = (f \times f)^\ast \Delta_k^Y$. Thus, by Theorem \ref{Ganea-Schwarz Approach}, this corollary follows immediately.
\end{proof}

\subsection{Basic Inequalities}

\begin{Proposition} \label{pullback category bounds}
Let $f:X \to Y$ be a map. Then
\begin{enumerate}
    \item $\pullTC{f} \leq \min \{ \TC{X}, \TC{Y} \}$;
    \item $\cat{f} \leq \pullTC{f} \leq \cat{f \times f}$.
\end{enumerate}
\end{Proposition}

Proposition 2.2 of \cite{Murillo_Wu} proves all of these inequalities with the exception of $\pullTC{f} \leq \TC{Y}$, but that inequality is relatively easy now that we have the perspective of pullbacks.

\begin{proof}
(1) Note that $\pullTC{f} \leq \TC{Y}$ is immediate by Proposition \ref{sectional pullback} since $\pullTC{f} = \secat{(f \times f)^\ast \pi^Y}$. Now suppose $\TC{X}=k$ and let $U_0,...,U_k$ be an open cover of $X \times X$ with motion planners $s_i:U_i \to X^I$. Define $f_i:U_i \to Y^I$ by $f_i = f_\ast s_i$. Then
\[f_i(x_0,x_1)(0)
= f(s_i(x_0,x_1)(0))
= f(x_0)\]
and similarly
\[f_i(x_0,x_1)(1)
= f(s_i(x_0,x_1)(1))
= f(x_1).\]
so that the $f_i$ are $f$-motion planners and $\pullTC{f} \leq k$.

(2) Note that if $X$ is empty, then $\cat{f} = \pullTC{f} = \cat{f \times f} = -1$, so it suffices to assume $X$ is nonempty.

First we prove $\cat{f} \leq \pullTC{f}$. Suppose $\pullTC{f}=k$ and let $U_0,...,U_k$ be an open cover of $X \times X$ with $f$-motion planners $f_i:U_i \to Y^I$. Let $b$ be a choice of base point for $X$ and define $V_i = \{ x \mid (b,x) \in U_i \}$. Then $V_0,...,V_k$ is clearly an open cover for $X$. Now we define a nullhomotopy $H_i:V_i \times I \to Y$ of $f$ on $V_i$ by
\[ H_i(x,t) = f(b,x)(t). \]
Therefore, $\cat{f} \leq k$.

Finally, we show that $\pullTC{f} \leq \cat{f \times f}$. Suppose $\cat{f \times f} = k$ and let $U_0,...,U_k$ be open sets covering $X \times X$ such that each $(f \times f)|_{U_i}$ is nullhomotopic. Let $H_i:U_i \times I \to Y \times Y$ be a nullhomotopy of $(f \times f)|_{U_i}$ to some $(y_0^i,y_1^i) \in Y \times Y$. Since $Y$ is path-connected, let $\gamma_i$ be some path from $y_0^i$ to $y_1^i$. Now define $f_i:U_i \to Y^I$ by
\[f_i(x_0,x_1)(t) =
\begin{cases}
p_0 \left( H_i(x_0,x_1,3t) \right) & \text{if } 0 \leq t \leq \frac{1}{3} \\
\gamma_i(3t-1) & \text{if } \frac{1}{3} \leq t \leq \frac{2}{3} \\
p_1 \left( H_i(x_0,x_1,3 - 3t) \right) &  \text{if }\frac{2}{3} \leq t \leq 1
\end{cases}\]
where $p_0,p_1:Y \times Y \to Y$ are projections into the first and second coordinates, respectively. Then $f_i$ is continuous since $H_i(x_0,x_1)(1)=(y_0^i,y_1^i)$. Furthermore,
\[ f_i(x_0,x_1)(0)
= p_0 \left( H_i(x_0,x_1,3 \cdot 0) \right)
= p_0 \left( (f \times f)(x_0,x_1) \right)
= f(x_0)\]
and
\[ f_i(x_0,x_1)(1)
= p_1 \left( H_i(x_0,x_1,3-3 \cdot 1) \right)
= p_1 \left( (f \times f)(x_0,x_1) \right)
= f(x_1)\]
so that the $f_i$ are $f$-motion planners. Therefore, $\pullTC{f} \leq k$.
\end{proof}

The following corollary corresponds to Corollary 2.3 of \cite{Murillo_Wu}:

\begin{Corollary} \label{pullback nullhomotopic}
Let $f:X \to Y$ be any map. Then $\pullTC{f} = 0$ if and only if $f$ is nullhomotopic.
\end{Corollary}

\begin{proof}
First suppose $\pullTC{f} = 0$. Then
\[ 0 \leq \cat{f} \leq \pullTC{f} = 0\]
so that $\cat{f} = 0$ and $f$ is nullhomotopic.

Now suppose $f$ is nullhomotopic. Then $f \times f$ is nullhomotopic so that
\[ 0 \leq \pullTC{f} \leq \cat{f \times f} = 0\]
and hence $\pullTC{f}=0$.
\end{proof}

\begin{Proposition}
Let $f:X \to Z$ and $g:Y \to W$ be any maps such that $X \times X$ and $Y \times Y$ are both normal. Then $\pullTC{f \times g} \leq \pullTC{f} + \pullTC{g}$.
\end{Proposition}

\begin{proof}
This is immediate by Proposition \ref{secat product}.
\end{proof}

The cup length lower bound was proven in \cite{Murillo_Wu} for CW-complexes with coefficients in a constant ring. With an identical proof, we can switch CW-complexes for $X$ an ANR and use coefficients in any $\pi_1(Y) \times \pi_1(Y)$ module:

\begin{Theorem} \label{pullback cohomology bound}
Let $f:X \to Y$ be any map with $X$ an ANR. Suppose that $u_i \in \widetilde{H}^\ast(X \times X; A_i)$ are such that $u_i \in \ker \Delta^\ast \cap {\rm im} (f \times f)^\ast$ and $u_0 \cup ... \cup u_k \neq 0$, where each $A_i$ are $\pi_1(Y) \times \pi_1(Y)$ modules. Then $\pullTC{f} \geq k + 1$. In other words,
\[\cuplength{\ker \Delta^\ast \cap {\rm im} (f \times f)^\ast} \leq \pullTC{f}.\]
\end{Theorem}

\begin{proof}
Suppose $\pullTC{f} \leq k$ so that there are open sets $U_0,...,U_k$ covering $X \times X$ such that on each $U_i$ there is a map $h:U_i \to X$ satisfying $(f \times f) \Delta h \simeq (f \times f)|_{U_i}$. Thus, for each $i$ we have a commutative diagram
\[\begin{tikzcd}
... \arrow[r]
& H^n \left( X \times X, U_i; A_i \right) \arrow[r, "q_i^\ast"]
& H^n \left( X \times X; A_i \right) \arrow[r, "j_i^\ast"]
& H^n \left( U_\ell; A_i \right) \arrow[r]
& ...
\\
&&& H^n \left( X; A_i \right) \arrow[u, swap, "h_i^\ast"]
\\
&& H^n \left( Y \times Y; A_i \right) \arrow[uu, "(f \times f)^\ast"] \arrow[r, swap, "(f \times f)^\ast"]
& H^n \left( X \times X; A_i \right) \arrow[u, swap, "\Delta^\ast"]
\end{tikzcd}\]
where the top row is the long exact sequence of the pair $(X \times X, U_i)$, and the maps $q_i$ and $j_i$ are the obvious ones.

Now let $u_0,...,u_k \in \ker \Delta^\ast \cap {\rm im} (f \times f)^\ast$ where each $u_i$ is taken in coefficients $A_i$ and has degree at least $1$. Then there is some $v_i \in \ker H^\ast(Y \times Y;A_i)$ such that $(f \times f)^\ast v_i = u_i$. Since each $u_i \in \ker \Delta$, it follows that
\[ j_i^\ast u_i
= j_i^\ast (f \times f)^\ast v_i
= h_i^\ast \Delta^\ast (f \times f)^\ast v_i
= h_i^\ast \Delta^\ast u_i
= h_i^\ast 0
= 0\]
so that $u_i \in \ker j_i^\ast$. By exactness, there is some $w_i \in H^\ast \left( X \times X, U_i; A_i \right)$ such that $q_i^\ast w_i = u_i$. Now consider $w_0 \cup ... \cup w_k$, which lies in
\[ H^\ast \left( X \times X, U_0 \cup ... \cup U_k; A_0 \otimes ... \otimes A_k \right)
= H^\ast \left( X \times X, X \times X; A_0 \otimes ... \otimes A_k \right)
= 0\]
so that $w_0 \cup ... \cup w_k = 0$; hence,
\[ u_0 \cup ... \cup u_k
= q_0^\ast w_0 \cup ... \cup q_k^\ast w_k
= q^\ast \left( w_0 \cup ... \cup w_k \right)
= 0\]
where $q:X \times X \to (X \times X,X \times X)$ is the obvious map of pairs. Thus, the theorem follows.
\end{proof}

\subsection{Homotopy Invariance}

\begin{Proposition}
If $f,g:X \to Y$ are homotopic, then $\pullTC{f} = \pullTC{g}$.
\end{Proposition}

\begin{proof}
By symmetry, it suffices to prove $\pullTC{g} \leq \pullTC{f}$. Suppose $\pullTC{f} = k$ so that there is an open covering $U_0,...,U_k$ of $X \times X$ such that each restriction $(f \times f)|_{U_i}$ can be deformed into $\Delta Y$. Since $g$ is homotopic to $f$, it follows that $(g \times g)|_{U_i}$ is homotopic to $(f \times f)|_{U_i}$; hence, $(g \times g)|_{U_i}$ can also be deformed into $\Delta Y$. Therefore, $\pullTC{g} \leq k = \pullTC{f}$.
\end{proof}

\begin{Proposition}
$\pullTC{gf} \leq \min \{ \pullTC{g}, \pullTC{f} \}$ for any maps $f:X \to Y$ and $g:Y \to Z$.
\end{Proposition}

\begin{proof}
Suppose that $\pullTC{f} = k$ so that there is an open cover $U_0,...,U_k$ of $X \times X$ such that each restriction $(f \times f)|_{U_i}$ can be deformed into $\Delta Y$ via some homotopy $F_i:U_i \times I \to Y \times Y$. Now define a homotopy $H_i:U_i \times I \to Z \times Z$ by $H_i = (g \times g)F_i$, which defines a deformation of $(gf \times gf)|_{U_i}$ into $\Delta Z$; hence, $\pullTC{gf} \leq k = \pullTC{f}$.

Now suppose that $\pullTC{g} = k$ so that there is an open cover $V_0,...,V_k$ of $Y \times Y$ such that each restriction $(g \times g)|_{V_i}$ can be deformed into $\Delta Z$ via some homotopy $G_i:V_i \times I \to Z \times Z$. Now let $U_i = (f \times f)^{-1}(V_i)$ and define a homotopy $H_i:U_i \times I \to Z \times Z$ by
\[H_i(x_0,x_1,t) = G_i(f(x_0),f(x_1),t),\]
which defines a deformation of $(gf \times gf)|_{U_i}$ into $\Delta Z$; hence, $\pullTC{gf} \leq k = \pullTC{g}$.
\end{proof}

\begin{Corollary} \label{pullback homotopy equivalence}
Let $h:X \to Y$ be any map.
\begin{enumerate}
    \item If $h$ has a right homotopy inverse, then $\pullTC{h} = \TC{Y}$.
    \item If $h$ has a left homotopy inverse, then $\pullTC{h} = \TC{X}$.
    \item If $h$ is a homotopy equivalence, then $\pullTC{h} = \TC{X} = \TC{Y}$.
\end{enumerate}
\end{Corollary}

\begin{proof}
{\bf (1)} Let $g:Y \to X$ be the right homotopy inverse of $h$. Then $hg$ is homotopic to $Id_Y$, so that $\pullTC{hg} = \pullTC{Id_Y}$. Thus, we have
\[\TC{Y}
= \pullTC{Id_Y}
= \pullTC{hg}
\leq \pullTC{h}
\leq \TC{Y}\]
so that $\pullTC{h} = \TC{Y}$.

{\bf (2)} Let $g:Y \to X$ be the left homotopy inverse of $h$. Then $gh$ is homotopic to $Id_X$, so that $\pullTC{gh} = \pullTC{Id_X}$. Thus, we have
\[\TC{X}
= \pullTC{Id_X}
= \pullTC{gh}
\leq \pullTC{h}
\leq \TC{X}\]
so that $\pullTC{h} = \TC{X}$.
\end{proof}

\begin{Proposition} \label{pullback fibration replacement}
Suppose we have a diagram of the form
\[ \begin{tikzcd}[column sep = 1em]
X \arrow[rr, shift left, "h"] \arrow[dr, swap, "f"]
&& X' \arrow[dl, "f'"]
\\
&Y
\end{tikzcd}\]
such that $h$ is a homotopy domination. Then $\pullTC{f} = \pullTC{f'}$.
\end{Proposition}

\begin{proof}
It's immediate that $\pullTC{f} = \pullTC{f'h} \leq \pullTC{f'}$, so it suffices to prove $\pullTC{f'} \leq \pullTC{f}$. Since $h$ is a homotopy domination, it has a right homotopy inverse, i.e., a map $g:X' \to X$ such that $hg$ is homotopic to $Id_{X'}$. Then $fg = f'hg$, so that $fg$ is homotopic to $f'$; hence, $\pullTC{f'} = \pullTC{fg} \leq \pullTC{f}$.
\end{proof}

\begin{Proposition} \label{pullback cofibration replacement}
Suppose we have a diagram of the form
\[ \begin{tikzcd}[column sep = 1em]
& X \arrow[dl, swap, "f"] \arrow[dr, "f'"]
\\
Y \arrow[rr, swap, "h"]
&& Y'
\end{tikzcd}\]
such that $h$ has a left homotopy inverse. Then $\pullTC{f} = \pullTC{f'}$.
\end{Proposition}

\begin{proof}
It's immediate that $\pullTC{f'} = \pullTC{hf} \leq \pullTC{f}$, so it suffices to prove $\pullTC{f} \leq \pullTC{f'}$. Let $g:Y' \to Y$ be the left homotopy inverse of $h$ so that $gh$ is homotopic to $Id_Y$. Then $gf' = ghf$, so that $gf'$ is homotopic to $f$; hence, $\pullTC{f} = \pullTC{gf'} \leq \pullTC{f'}$.
\end{proof}

\begin{Remark}
Propositions \ref{pullback fibration replacement} and \ref{pullback cofibration replacement} imply that $f$ can always be replaced by either a fibration or a cofibration for the purposes of computing $\pullTC{f}$.
\end{Remark}

\subsection{Examples}

\begin{Examples} Fiber bundles of projective spaces:
\begin{enumerate}
    \item Let $p:S^n \to \mathbb{R}P^n$ be the obvious covering map. Consider the sets
    \[ F_0 = \{ (x,y) \in S^n \times S^n \mid x \neq -y \}\]
    and
    \[ F_1 = \{ (x,-x) \mid x \in S^n \},\]
    i.e., the sets of nonantipodal and antipodal pairs respectively. On the set of nonantipodal pairs, we can map pairs $(x,y)$ to the unique geodesic from $x$ to $y$ and compose that with $p$ to get a $p$-motion planner $p_0:F_0 \to (\mathbb{R}P^n)^I$. Since antipodal pairs map to the same equivalence class in $\mathbb{R}P^n$, it follows that mapping $(x,-x)$ to the constant path at $p(x)$ is a $p$-motion planner on $F_1$. Thus, $\pullTC{p} = 1$ since $p$ is not nullhomotopic.
    \item $\pullTC{p:S^{2n+1} \to \mathbb{C}P^n} = 1$ since $\pullTC{p} \leq \TC{S^{2n+1}} = 1$ and $p$ is not nullhomotopic.
    \item Similarly, $\pullTC{S^{4n+3} \to \mathbb{H}P^n} = 1$ and $\pullTC{S^{8n+7} \to \mathbb{O}P^n} = 1$.
\end{enumerate}
\end{Examples}

\begin{Definition} \label{h space definition}
An H-space is a topological space $X$ with a map $\mu:X \times X \to X$ and an element $e \in X$ such that $\mu i_0$ and $\mu i_1$ are homotopic to the identity map $Id_X$, where $i_0,i_1:X \to X \times X$ are the inclusion maps given by $i_0(x)=(x,e)$ and $i_1(x)=(e,x)$.
\end{Definition}

For an H-space $X$, there is an equality $\TC{X} =\cat{X}  $. This is proven in \cite{Garcia-Calcines_Garcia-Calcines_Vandembroucq} and also in \cite{Farber_3} for the special case of topological groups. We extend this theorem to maps when the domain is an H-space:

\begin{Theorem} \label{H-space}
Let $f:X \to Y$ be any map such that $X$ is an {\rm H}-space. Then $\pullTC{f} = \cat{f}$.
\end{Theorem}

\begin{proof}
Let $\mu, i_0, i_1$, and $e$ be as in definition \ref{h space definition} for the H-space $X$. Note that it is sufficient to prove $\pullTC{f} \leq \cat{f}$, so let $\cat{f} = k$ and let $U_0,...,U_k$ be an open cover of $X$ such that each restriction $f|_{U_i}$ is nullhomotopic. Define $V_i$ by $V_i = \mu^{-1}(U_i)$, let $H_i:X \times I \to Y$ be a nullhomotopy of $f|_{U_i}$, and let $F_0,F_1:X \times I \to X$ be homotopies of $\mu i_0$ to $Id_X$ and of $\mu i_1$ to $Id_X$ respectively. Now define an $f$-motion planner $f_i:V_i \to Y^I$ on $V_i$ by
\[ f_i(x_0,x_1)(t) =
\begin{cases}
H_i(F_0(x_0,1-2t),2t) & 0 \leq t \leq \frac{1}{2} \\
H_i(F_1(x_1,2t-1),2-2t) & \frac{1}{2} \leq t \leq 1
\end{cases}\]
so that $\pullTC{f} \leq \cat{f}$.
\end{proof}

\section{The Mixed TC of a Map}

In a series of three papers \cite{Pavesic_1,Pavesic_2,Pavesic_3}, the (reduced) topological complexity of a map $f:X \to Y$ is defined to be the least integer $n$ such that there is a sequence of closed sets
\[ \emptyset = C_{-1} \subset C_0 \subset ... \subset C_n = Y \times X\]
such that each difference $C_i \setminus C_{i-1}$ admits a partial section of the map $\pi_f:X^I \to Y \times X$ given by $\alpha \mapsto \left( f(\alpha(0)),\alpha(0) \right)$. In this paper, we will call this invariant the {\it robotic topological complexity} of the map $f$, denoted $\robotTC{f}$, in order to differentiate it from the definition in the previous section. This name is inspired by the fact that it was initially defined for applications involving robotic arms.

\begin{Theorem} [\cite{Pavesic_3}] \label{pavesic fibration}
Let $f:X \to Y$ be a map. Then $\pi_f$ is a fibration if and only if $f$ is a fibration. Moreover, if $f$ is a fibration, then $\robotTC{f} = \secat{\pi_f} = \secat{(Id_Y \times f)^\ast \Delta_0^Y}$.
\end{Theorem}

In light of this theorem, the main focus of this section is to study a notion for the topological complexity of a map using the following pullback diagram:
\[ \begin{tikzcd}[column sep = 4em, row sep = 4em]
(Id_Y \times f)^\ast Y^I \arrow[r, "\overline{Id_Y \times f}"] \arrow[d, swap, "(Id_Y \times f)^\ast \Delta_0^Y"]
& Y^I \arrow[d, "\Delta_0^Y"]
\\
Y \times X \arrow[r, swap, "Id_Y \times f"]
& Y \times Y
\end{tikzcd}\]

\subsection{Definitions}

\begin{Definition}
Let $f:X \to Y$ be a map.
\begin{enumerate}
    \item A {\it mixed $f$-motion planner} on a subset $Z \subset Y \times X$ is a map $f_{\subsize{Z}}:Z \to Y^I$ such that $f_{\subsize{Z}}(y,x)(0)=y$ and $f_{\subsize{Z}}(x_0,x_1)(1)=f(x_1)$.
    \item A {\it mixed $f$-motion planning algorithm} is a cover of $Y \times X$ by sets $Z_0,...,Z_k$ such that on each $Z_i$ there is some mixed $f$-motion planner $f_i:Z_i \to Y^I$.
    \item The {\it mixed topological complexity} of $f$, denoted $\halfTC{f}$, is the least $k$ such that $Y \times X$ can be covered by $k+1$ open subsets $U_0,...,U_k$ on which there are mixed $f$-motion planners. If no such $k$ exists, we define $\halfTC{f} = \infty$.
\end{enumerate}
\end{Definition}

\begin{Remark}
If $Y$ isn't path connected, $\halfTC{f}$ will always be $\infty$, so to avoid this one would need to add up $\halfTC{f}$ for each path component of $Y$. Thus, as with the previous section, we assume all spaces to be path-connected.
\end{Remark}

\begin{Examples}
{\color{white} LATEX SPACING ISSUE!!!}
\begin{enumerate}
    \item $\halfTC{f} = -1$ if and only if $f$ is the empty map.
    \item If $i:A \to X$ is an inclusion map, then $\halfTC{i} = \relTC{X}{X \times A}$.
    \item $\halfTC{Id_X} = \TC{X}$ for any space $X$.
\end{enumerate}
\end{Examples}

As with (pullback) TC, there are versions of Theorems \ref{pullback definitions} and \ref{pullback cover ANR} and Corollary \ref{pullback TC lifting} for mixed TC as well, though we will skip the proofs as they are identical

\begin{Theorem} \label{mixed definitions}
Let $f:X \to Y$ be a map and let $Z \subset Y \times X$. The following are equivalent:
\begin{enumerate}
    \item there is a section $s:Z \to (Id_Y \times f)^\ast Y^I$ of the pullback fibration $(Id_Y \times f)^\ast \Delta_0^Y$;
    \item there is a mixed $f$-motion planner $f_{\subsize{Z}}:Z \to Y^I$;
    \item $(Id_Y \times f)|_Z$ can be deformed into $\Delta Y$.
\end{enumerate}
\end{Theorem}

\begin{Theorem} \label{mixed cover ANR}
If $f:X \to Y$ is a map between ANRs, then $\halfTC{f}=\secatg{(Id_Y \times f)^\ast \Delta^Y_0}$.
\end{Theorem}

\begin{Corollary} \label{mixed TC lifting}
Let $f:X \to Y$ be any map with $Y \times X$ normal. Then $\halfTC{f} \leq k$ if and only if there is a lift of $Id \times f$ with respect to $\Delta^Y_{k+1}$.
\end{Corollary}

\subsection{Basic Inequalities}

\begin{Proposition} \label{category bounds}
Let $f:X \to Y$ be a map. Then
\begin{enumerate}
    \item $\cat{Y} \leq \halfTC{f} \leq  \TC{Y}$;
    \item $\cat{f} \leq \halfTC{f} \leq \cat{Id_Y \times f}$.
\end{enumerate}
\end{Proposition}

\begin{proof}
The proofs for $(2)$ and for the inequality $\halfTC{f} \leq \TC{Y}$ are the same as in Proposition \ref{pullback category bounds}. Then the proof of $\cat{Id_Y} \leq \halfTC{f}$ is the same as the proof of $\cat{f} \leq \halfTC{f}$, so that $\cat{Y} \leq \halfTC{f}$.
\end{proof}

\begin{Corollary} \label{half nullhomotopic}
Let $f:X \to Y$ be any map. If $f$ is nullhomotopic, then $\halfTC{f} = \cat{Y}$.
\end{Corollary}

\begin{proof}
First suppose that $f$ is nullhomotopic and let $\cat{Y}=k$. Then there is a cover $U_0,...,U_k$ of $Y$ on which $Id_Y$ is nullhomotopic. Therefore, $U_i \times X$ is a cover of $Y \times X$ on which $Id_Y \times f$ is nullhomotopic; hence, $\cat{Id_Y \times f} \leq k = \cat{Y}$. Therefore,
\[ \cat{Y} \leq \halfTC{f} \leq \cat{Id_Y \times f} \leq \cat{Y}\]
so that $\cat{Y} = \halfTC{f}$.
\end{proof}

The converse of the above doesn't hold as demonstrated by the following example:

\begin{Example}
Let $Y$ be a non-contractible topological group, e.g., $S^1$. Then we have $\halfTC{Id_Y} = \TC{Y} = \cat{Y}$, but $Id_Y$ isn't nullhomotopic since $Y$ isn't contractible.
\end{Example}

\begin{Corollary}
Let $f:X \to Y$ be any map. Then $\halfTC{f} = 0$ if and only if $Y$ is contractible.
\end{Corollary}

\begin{proof}
If $\halfTC{f} = 0$, then $\cat{Y} \leq \halfTC{f} = 0$ so that $Y$ is contractible. Conversely, if $Y$ is contractible, then $f$ is nullhomotopic so that $\halfTC{f} = \cat{Y} = 0$.
\end{proof}

\begin{Proposition}
Let $f:X \to Z$ and $g:Y \to W$ be any maps such that $Z \times X$ and $W \times Y$ are both normal. Then $\pullTC{f \times g} \leq \pullTC{f} + \pullTC{g}$.
\end{Proposition}

\begin{proof}
This is immediate by Proposition \ref{secat product}.
\end{proof}

\subsection{Homotopy Invariance}

\begin{Proposition}
If $f,g:X \to Y$ are homotopic, then $\halfTC{f} = \halfTC{g}$.
\end{Proposition}

\begin{proof}
This follows immediately from property (3) of Theorem \ref{mixed definitions}.
\end{proof}

\begin{Proposition}
$\halfTC{gf} \leq \halfTC{g}$ for any maps $f:X \to Y$ and $g:Y \to Z$.
\end{Proposition}

\begin{proof}
Let $\halfTC{g} = k$ so that there is an open cover $V_0,...,V_k$ of $Z \times Y$ such that each restriction $(Id_Z \times g)|_{V_i}$ can be deformed into $\Delta Z$ via some homotopy $G_i:V_i \times I \to Z \times Z$. Now let $U_i = (Id_Z \times f)^{-1}(V_i)$ and define a homotopy $H_i:U_i \times I \to Z \times Z$ by
\[H_i(z,x,t) = G_i(z,f(x),t),\]
which defines a deformation of $(Id_Z \times gf)|_{U_i}$ into $\Delta Z$; hence, $\halfTC{gf} \leq k = \halfTC{g}$.
\end{proof}

Unfortunately, unlike the (pullback) TC of a map, it is possible for $\halfTC{f}$ to be smaller than $\halfTC{gf}$:

\begin{Example}
Consider any maps $f:X \to Y$ and $g:Y \to Z$ such that $f$ is nullhomotopic and $\cat{Y} < \cat{Z}$. Then $gf$ is also nullhomotopic so that $\halfTC{f} = \cat{Y}$ and $\halfTC{gf} = \cat{Z}$ by Corollary \ref{half nullhomotopic}; hence, $\halfTC{f} < \halfTC{gf}$. E.g., take $X = Y = \ast$ and $Z = S^1$ with $f$ and $g$ the obvious maps so that $\halfTC{f} = 0$ and $\halfTC{gf} = 1$.
\end{Example}

\begin{Corollary}
Let $h:X \to Y$ be any map.
\begin{enumerate}
    \item If $h$ has a right homotopy inverse, then $\halfTC{h} = \TC{Y}$.
    \item If $h$ is a homotopy equivalence, then $\halfTC{h} = \TC{X} = \TC{Y}$.
\end{enumerate}
\end{Corollary}

\begin{proof}
{\bf (1)} Let $g:Y \to X$ be the right homotopy inverse of $h$. Then $hg$ is homotopic to $Id_Y$, so that $\halfTC{hg} = \halfTC{Id_Y}$. Thus, we have
\[\TC{Y}
= \halfTC{Id_Y}
= \halfTC{hg}
\leq \halfTC{h}
\leq \TC{Y}\]
so that $\halfTC{h} = \TC{Y}$.
\end{proof}

Unlike Corollary \ref{pullback homotopy equivalence} for $\pullTC{h}$, it does not follow that $\halfTC{h} = \TC{X}$ when $h$ has a left homotopy inverse. See the following counterexample:

\begin{Example} \label{nullhomotopic example}
Let $h:\ast \to S^1$ be an inclusion map. Then $h$ has the obvious left (homotopy) inverse $g:S^1 \to \ast$. Since $h$ is nullhomotopic, $\halfTC{h} = \cat{S^1} = 1$ but unfortunately $\TC{\ast} = 0 < \halfTC{h}$.
\end{Example}

\begin{Proposition} \label{half fibration replacement}
Suppose we have a diagram of the form
\[ \begin{tikzcd}[column sep = 1em]
X \arrow[rr, shift left, "h"] \arrow[dr, swap, "f"]
&& X' \arrow[dl, "f'"]
\\
&Y
\end{tikzcd}\]
such that $h$ is a homotopy domination. Then $\halfTC{f} = \halfTC{f'}$.
\end{Proposition}

\begin{proof}
It's immediate that $\halfTC{f} = \halfTC{f'h} \leq \halfTC{f'}$, so it suffices to prove $\halfTC{f'} \leq \halfTC{f}$. Since $h$ is a homotopy domination, it has a right homotopy inverse, i.e., a map $g:X' \to X$ such that $hg$ is homotopic to $Id_{X'}$. Then $fg = f'hg$, so that $fg$ is homotopic to $f'$; hence, $\halfTC{f'} = \halfTC{fg} \leq \halfTC{f}$.
\end{proof}

\begin{Proposition} \label{half cofibration replacement}
Suppose we have a diagram of the form
\[ \begin{tikzcd}[column sep = 1em]
& X \arrow[dl, swap, "f"] \arrow[dr, "f'"]
\\
Y \arrow[rr, swap, "h"]
&& Y'
\end{tikzcd}\]
such that $h$ is a homotopy equivalence. Then $\halfTC{f} = \halfTC{f'}$.
\end{Proposition}

\begin{proof}
Let $g:Y' \to Y$ be the homotopy inverse of $h$ and note that the commutative triangle above induces the following commutative square:
\[ \begin{tikzcd}[column sep = 3em, row sep = 3em]
Y \times X \arrow[r, "h \times Id_X"] \arrow[d, swap, "Id_Y \times f"]
& Y' \times X \arrow[d, "Id_{Y'} \times f'"]
\\
Y \times Y \arrow[r, swap, "h \times h"]
& Y' \times Y'
\end{tikzcd}\]

First let $\halfTC{f} = k$ and let $U_0,...,U_k$ be an open cover of $Y \times X$ such that each restriction $(Id_Y \times f)|_{U_i}$ can be deformed into $\Delta Y$ via some homotopy $H_i:U_i \to Y \times Y$. Define an open cover $V_i = (g \times Id_X)^{-1}(U_i)$ of $Y' \times X$. Then it follows that
\[ Id_{Y'} \times f'
\simeq hg \times f'
= hg \times hf.\]
Now define a deformation $G_i:V_i \to Y \times Y$ of $(g \times f)|_{V_i}$ into $\Delta Y$ given by $G_i(y',x,t) = H_i(g(y'),x,t)$; hence, $(hg \times hf)|_{V_i}$ can be deformed into $\Delta Y'$. Therefore, $\halfTC{f'} \leq k$.

Now let $\halfTC{f'} = k$ and let $V_0,...,V_k$ be an open cover of $Y' \times X$ such that each restriction $(Id_{Y'} \times f')|_{V_i}$ can be deformed into $\Delta Y'$ via some homotopy $G_i:V_i \to Y' \times Y'$. Define an open cover $U_i = (h \times Id_X)^{-1}(V_i)$ of $Y \times X$. Then it follows that
\[ Id_Y \times f
\simeq gh \times ghf
= gh \times gf'.\]
Now define a deformation $H_i:U_i \to Y' \times Y'$ of $(h \times f')|_{U_i}$ into $\Delta Y'$ given by $H_i(y,x,t) = G_i(h(y),x,t)$; hence, $(gh \times gf')|_{U_i}$ can be deformed into $\Delta Y$. Therefore, $\halfTC{f} \leq k$.
\end{proof}

\begin{Remark}
Propositions \ref{half fibration replacement} and \ref{half cofibration replacement} imply that $f$ can always be replaced by either a fibration or a cofibration for the purposes of computing $\halfTC{f}$. In fact, many of the bounds for mixed TC that have been presented here can also be obtained by taking a fibration replacement and then applying the results of \cite{Pavesic_3}.
\end{Remark}

\subsection{Cohomological Lower Bound}

\begin{Theorem}
Let $f:X \to Y$ be any map such that $X \times Y$ is an ANR. Suppose that $u_i \in \widetilde{H}^\ast(Y \times X; A_i)$ are such that $u_i \in \ker (f,Id_X)^\ast$ and $u_0 \cup ... \cup u_k \neq 0$, where each $A_i$ are $\pi_1(Y) \times \pi_1(X)$ modules. Then $\halfTC{f} \geq k + 1$. In other words,
\[\cuplength{\ker (f,Id_X)^\ast} \leq \halfTC{f}.\]
\end{Theorem}

\begin{proof}
First note that we can assume $f$ is a fibration by taking a fibration replacement and applying Proposition \ref{half fibration replacement} so that $\pi_f$ is a fibration with $\secat{\pi_f} = \halfTC{f} = \robotTC{f}$ by Theorem \ref{pavesic fibration}. From here one need only apply the usual proof for the cup product lower bound for the sectional category of the fibration $\pi_f$ (see \cite{Schwarz}), or one could note that the cup product lower bound in \cite{Pavesic_3} also works more generally for ANRs and coefficients in $\pi_1(Y) \times \pi_1(X)$ modules.
\end{proof}

\subsection{Comparisons}

So far this paper has mentioned 3 potential notions for the topological complexity of a map, two of which are the homotopy invariant notions that are studied in depth in sections 3 and 4. These notions are defined as pullbacks, but one might also consider relative notions such as $\TC{C_f}$ where $C_f$ is the mapping cone of $f$; or $\relTC{M_f}{X \times X}$ where $M_f$ is the mapping cylinder; or the relative topological complexity of the pair $\pairTC{M_f}{X}$, whose general notion was introduced in \cite{Short}.

\begin{Example}
If $f = Id_{S^1}$, then $C_f$ is contractible so that $\TC{C_f}=0$, but $\TC{S^1}=1$.
\end{Example}

This shows that $\TC{C_f}$ wouldn't be a generalization of the topological complexity of a space, so it isn't considered here.

\begin{Definition}
Let $(X,A)$ with $A \subset X$ be a pair of topological spaces. Then the relative topological complexity of $(X,A)$, denoted $\pairTC{X}{A}$, is defined to be $\secat{(i \times Id_X)^\ast \Delta_0^X}$ where $i:A \to X$ is the inclusion map. (See \cite{Short} for a full introduction of this invariant.)
\end{Definition}

\begin{Remark}
By Theorem \ref{mixed definitions}, $\pairTC{X}{A}$ is the same as $\halfTC{i}$.
\end{Remark}

We now have a list of reasonable definitions for the topological complexity of a map:
\begin{enumerate}
    \item $\pullTC{f}$
    \item $\relTC{M_f}{X \times X}$
    \item $\halfTC{f}$
    \item $\pairTC{M_f}{X}$
    \item $\robotTC{f}$
\end{enumerate}
The following theorem describes their relationships:

\begin{Theorem}
Let $f:X \to Y$ be any map. Then
\[ \pullTC{f} = \relTC{M_f}{X \times X} \leq \halfTC{f} = \pairTC{M_f}{X} \leq \robotTC{f}.\]
\end{Theorem}

\begin{proof}
First note that Propositions \ref{pullback cofibration replacement} and \ref{half cofibration replacement} allow us to replace $f$ with the cofibration $i:X \to M_f$ for the purposes of computing $\pullTC{f}$ and $\halfTC{f}$ respectively. Therefore, $\pullTC{f} = \relTC{M_f}{X \times X}$ and $\halfTC{f} = \pairTC{M_f}{X}$.

The inequality $\halfTC{f} \leq \robotTC{f}$ follows by taking a fibration replacement $p_f$ of $f$ and applying Theorem \ref{pavesic fibration} to get that $\halfTC{f} = \halfTC{p_f} = \robotTC{p_f}$. Then \cite{Pavesic_3} proves that fibration replacements give the inequality $\robotTC{p_f} \leq \robotTC{f}$.

All that remains is to prove $\pullTC{f} \leq \halfTC{f}$. Note that $\pullTC{f}$ and $\halfTC{f}$ are defined as the sectional categories of $(f \times f)^\ast \Delta_0^Y$ and $(Id_Y \times f)^\ast \Delta_0^Y$ respectively. Furthermore, it follows that
\[ (f \times f)^\ast \Delta_0^Y
= (f \times Id_X)^\ast (Id_Y \times f)^\ast \Delta_0^Y \]
so that $\pullTC{f} \leq \halfTC{f}$ by Proposition \ref{sectional pullback}.
\end{proof}

\begin{Corollary}
Let $f:X \to Y$ be a fibration. Then $\halfTC{f} = \pairTC{M_f}{X} = \robotTC{f}$.
\end{Corollary}

\begin{proof}
This is immediate from the previous theorem and Theorem \ref{pavesic fibration}.
\end{proof}

In \cite{Pavesic_3}, it's shown that $\halfTC{f}$ and $\robotTC{f}$ can be arbitrarily far apart when $f$ is not a fibration. Example \ref{nullhomotopic example} also shows that $\pullTC{f}$ and $\halfTC{f}$ can differ, so that each of these three notions of TC are distinct. In fact, it seems to be the case that $\halfTC{f}$ is significantly more dependent on the codomain than $\pullTC{f}$, e.g., $\halfTC{f:X \to T^n} = n$ independent of $X$ or $f$.

\section{Monoidal TC of a Map}

This section will show how $f$-motion planners can be used to extend the definition of monoidal TC to maps. A similar adjustment can also be used to define the symmetric TC of maps, but that will be omitted for purposes of this work.

For this section, denote by $\Delta_f$ the set $(f \times f)^{-1}(\Delta Y)$.

\begin{Definition}
Let $f:X \to Y$ be any map.
\begin{enumerate}
    \item A {\it reserved $f$-motion planner} on a subset $Z \subset X \times X$ is an $f$-motion planner $f_Z$ on $Z$ such that $f_Z|_{\Delta_f} = i_{\subsize{Y}} (f \times f)|_{Z \cap \Delta_f}$.
    \item A {\it reserved $f$-motion planning algorithm} is a cover of $X \times X$ by sets $Z_0,...,Z_k$ such that on each $Z_i$ there is some reserved $f$-motion planner $f_i:Z_i \to Y^I$.
    \item The {\it monoidal topological complexity} of $f$, denoted $\MTC{f}$, is the least $k$ such that $X \times X$ can be covered by $k+1$ open subsets $U_0,...,U_k$ on which there are reserved $f$-motion planners. If no such $k$ exists, then define $\MTC{f} = \infty$.
\end{enumerate}
\end{Definition}

Observe that we recover the usual definition $\MTC{Id_X} = \MTC{X}$ for any space $X$. Furthermore, $\pullTC{f} \leq \MTC{f}$ for any map $f:X \to Y$. Thus, $\MTC{f}$ is a good candidate for generalizing the modoidal topological complexity of a space. As we'll see in what follows, many of the basic facts for spaces generalize to maps as well.

Unfortunately, monoidal TC isn't a homotopy invariant for maps since the counterexample for spaces applies to $\MTC{Id_X}$ (see \cite{Garcia-Calcines}). Although this means it can't be characterized as sectional category, like regular TC is, the reserved $f$-motion planners can still be characterized by a deformation property as is done for spaces in \cite{Dranishnikov_1}.

\begin{Theorem}
Let $f:X \to Y$ be a map and let $Z \subset X \times X$. The following are equivalent:
\begin{enumerate}
    \item there is a reserved $f$-motion planner $f_Z:Z \to Y^I$;
    \item $(f \times f)|_Z$ can be deformed into $\Delta Y$ rel $\Delta_f \cap Z$.
\end{enumerate}
\end{Theorem}

\begin{proof}
($\bf 1 \implies 2$) Let $f_Z:Z \to Y^I$ be a reserved $f$-motion planner. Now define a homotopy $H:X \times X \times I \to Y \times Y$ by
\[ H(x_0,x_1,t) = \left( f_Z(x_0,x_1)(t),f(x_1) \right),\]
which is a deformation of $(f \times f)|_Z$ into $\Delta Y$ rel $\Delta_f \cap Z$.

($\bf 2 \implies 1$) Now let $H:Z \times I \to Y \times Y$ be a deformation of $(f \times f)|_Z$ into $\Delta Y$ rel $\Delta_f \cap Z$. Define a map $f_Z:Z \to Y^I$ by
\[ f_Z(x_0,x_1)(t) =
\begin{cases}
\text{proj}_0 \left(H(x_0,x_1,2t) \right) & \text{if } 0 \leq t \leq \frac{1}{2} \\
\text{proj}_1 \left(H(x_0,x_1,2 - 2t) \right) & \text{if } \frac{1}{2} \leq t \leq 1
\end{cases}\]
where $\text{proj}_0,\text{proj}_1:Y \times Y \to Y$ are the projection maps into the corresponding coordinates. Then $f_Z$ is well-defined and continuous since $H(x_0,x_1,t) \in \Delta Y$ for all $(x_0,x_1) \in Z$. Moreover,
\[ f_Z(x_0,x_1)(0)
= \text{proj}_0 \left(H(x_0,x_1,0) \right)
= \text{proj}_0 \left(f(x_0),f(x_1)) \right)
= f(x_0)\]
and
\[ f_Z(x_0,x_1)(1)
= \text{proj}_1 \left(H(x_0,x_1,2-2) \right)
= \text{proj}_1 \left(H(x_0,x_1,0) \right)
= \text{proj}_1 \left(f(x_0),f(x_1)) \right)
= f(x_1).\]
Therefore, $f_Z$ is an $f$-motion planner on $Z$. Furthermore, if $f(x_0) = f(x_1)$ for $(x_0,x_1) \in Z$, then $H(x_0,x_1,t)=f(x_0)$ since $H$ is a deformation rel $\Delta_f \cap Z$; hence, $\overline{f}$ is reserved.
\end{proof}

In \cite{Dranishnikov_1} and \cite{Iwase_Sakai}, it was shown for spaces that $\TC{X} \leq \MTC{X} \leq \TC{X} + 1$ when $X$ is an ENR. This generalizes to maps as follows:

\begin{Theorem} \label{monoidal_inequality}
Let $f:X \to Y$ be any map with $X$ an ENR and $Y$ Hausdorff. Then $\pullTC{f} \leq \MTC{f} \leq \pullTC{f} + 1$.
\end{Theorem}

\begin{proof}
Let $\pullTC{f} = k$ and let $U_0,...,U_k$ be an open cover of $X \times X$ such that each restriction $(f \times f)|_{U_i}$ has a deformation $H_i:U_i \times I \to Y \times Y$ into $\Delta Y$. Consider the open sets $V_i = U_i \setminus \Delta_f$, on which there is a deformation of $(f \times f)|_{V_i}$ rel $\Delta_f \cap V_i = \emptyset$, namely $H_i|_{V_i}$. Note that $\Delta_f$ is an ENR since $X$ is an ENR and $Y$ is Hausdorff; hence, there is an open neighborhood $W$ of $\Delta_f$ and a homotopy $H:W \times I \to X \times X$ such that $H(-,0)$ is the inclusion map and $H(-,1)$ is a retraction. Then $(f \times f)H$ is a deformation of $(f \times f)|_W$ into $\Delta Y$ rel $\Delta_f$. Thus, $ \MTC{f} \leq \pullTC{f} + 1$ since $V_0,...,V_k,W$ is an open cover of $X \times X$.
\end{proof}

\begin{Proposition} \label{quotient_category}
Let $f:X \to Y$ be any map and consider the diagram:
\[\begin{tikzcd}[column sep = 4em, row sep = 4em]
X \times X \arrow[r, "f \times f"] \arrow[d, swap, "q_{\scaleto{X}{3 pt}}"]
& Y \times Y \arrow[d, "q_{\scaleto{Y}{3 pt}}"]
\\
(X \times X) / \Delta X \arrow[r, swap, "f_\Delta"]
& (Y \times Y) / \Delta Y
\end{tikzcd}\]
where $q_{\subsize{X}}$ and $q_{\subsize{Y}}$ are the obvious quotient maps and $f_\Delta$ is the induced map on the quotient spaces. Then $\cat{q_{\subsize{Y}}(f \times f)} \leq \pullTC{f}$ and $\cat{f_\Delta} \leq \MTC{f}$.
\end{Proposition}

\begin{proof}
Let $\pullTC{f} = k$ so that there is an open cover $U_0,...,U_k$ of $X \times X$ such that each restriction $(f \times f)|_{U_i}$ can be deformed into $\Delta Y$ via some homotopy $H_i:U_i \times I \to Y \times Y$. Now define $\overline{H}_i: X \times X \to (Y \times Y) / \Delta Y$ by $\overline{H}_i = q_{\subsize{Y}} H_i$, which is clearly a nullhomotopy of $q_{\subsize{Y}} (f \times f)$ on $U_i$; hence, $\cat{q_{\subsize{Y}}(f \times f)} \leq \pullTC{f}$.

Now let $\MTC{f} = k$. Then there are open sets $U_0,...,U_k$ such that each restriction $(f \times f)|_{U_i}$ can be deformed into $\Delta Y$ rel $\Delta_f \cap U_i$. Since the deformation on $(f \times f)|_{U_i}$ is done rel $\Delta_f \cap U_i$, it induces a nullhomotopy on each $f_\Delta|_{q_{\scaleto{X}{2.5 pt}}(U_i)}$ so that $\cat{f_\Delta} \leq k$.
\end{proof}

\begin{Corollary}
Let $f:X \to Y$ be any map with $X$ an ENR and $Y$ Hausdorff. Then 
\[ \cat{f_\Delta} \leq \pullTC{f} + 1.\]
\end{Corollary}

\begin{proof}
This is immediate by combining Theorem \ref{monoidal_inequality} and Proposition \ref{quotient_category}.
\end{proof}

\section{Discrete Group Homomorphisms}

\begin{Definition}
Given a group $\pi$, an Eilenberg-MacLane space $K(\pi,1) = B\pi$ is defined to be a space such that $\pi_1(B\pi) = \pi$ and $\pi_k(B\pi) = 0$ for $k \neq 1$.
\end{Definition}

Given $\pi$, an Eilenberg-MacLane space $B\pi = K(\pi,1)$ is unique up to homotopy equivalence so that we can define $\cat{\pi} = \cat{B\pi}$ and $\TC{\pi} = \TC{B\pi}$. The case of LS-category is known:

\begin{Theorem}[\cite{Eilenberg_Ganea}]
If $\pi$ is a group, then $\cat{\pi} = \cd{\pi}$, where $\cd{\pi}$ is the cohomological dimension of $\pi$.
\end{Theorem}
Since the cohomological dimension of groups with torsion is infinite, we  consider only torsion free groups in this paper.

Unfortunately, the case of topological complexity is more complicated since $\TC{\pi}$ can be anywhere between $\cd{\pi}$ and $2\cd{\pi}$. If $\pi$ is abelian, then $B\pi$ is a topological group so that $\TC{\pi} = \cd{\pi}$. If $\pi$ is a finitely generated torsion free hyberbolic group, then $\TC{\pi} = 2 \cd{\pi}$ (see \cite{Dranishnikov_2} and \cite{Farber_Mescher}).

The goal of this section is to introduce the LS-category and topological complexity of a group homomorphism. The former seems to have not been considered yet in the literature, which is likely because the LS-category of a group is known, but it is helpful in computing the latter.

Due to the homotopy invariance of Eilenberg MacLane spaces, there is a one-to-one correspondence between homomorphisms $f:G \to H$ and homotopy classes of continuous maps $Bf:BG \to BH$ that induce $f$ on the fundamental group. Therefore, the following definitions are well-defined:

\begin{Definition}
Let $f:G \to H$ be a group homomorphism.
\begin{enumerate}
    \item The LS-category of $f$, denoted $\cat{f}$, is defined to be $\cat{Bf}$.
    \item The topological complexity of $f$, denoted $\pullTC{f}$, is defined to be $\pullTC{Bf}$.
\end{enumerate}
\end{Definition}

\subsection{LS-Category}

Let $\pi$ be a group, and $I(\pi) \subset \mathbb{Z} [\pi]$ its augmentation ideal. Then 
\[ H^1(\pi;I(\pi)) \simeq \text{Hom}_{\mathbb{Z}[\pi]} \left( I(\pi), I(\pi) \right).\]
The {\it Berstein-Schwarz} class for $\pi$, denoted $\beta_\pi$, is defined to be the cohomology class on the left hand side that corresponds to $Id_{\mathbb{Z}[\pi]}$ on the right hand side.

\begin{Theorem}[\cite{Dranishnikov_Rudyak,Schwarz}]
Let $\pi$ be a group. Then 
\[\cd{\pi} = \max \{ k \mid \beta_\pi^k \neq 0 \}.\]
\end{Theorem}

In light of this fact, comes the following conjecture preceded by a useful lemma:

\begin{Lemma} \label{Berstein-Schwarz Lemma}
Let $f:G \to H$ be a homomorphism between discrete groups. If $f^\ast \beta_{\subsize{H}}^k = 0$, then $f^\ast u = 0$ for all $u \in H^k(H;A)$ where $A$ is any $\mathbb{Z}H$-module.
\end{Lemma}

\begin{proof}
Note that there is a $\mathbb{Z}H$-module homomorphism $\mu:I(H)^k \to A$ such that $\mu_\ast \beta_{\subsize{H}}^k = u$ (see Proposition 34 of \cite{Schwarz} and Corollary 3.5 of \cite{Dranishnikov_Rudyak}). Therefore,
\[ f^\ast u
= f^\ast \mu_\ast \beta_{\subsize{H}}^k
= \mu_\ast f^\ast \beta_{\subsize{H}}^k
= \mu_\ast 0
= 0.\]
\end{proof}

\begin{Conjecture} \label{LS Conjecture}
Let $f:G \to H$ be a homomorphism between discrete groups. Then 
\[\cat{f} = \max \{ k \mid f^\ast \beta_{\subsize{H}}^k \neq 0 \}.\]
\end{Conjecture}

Unfortunately, all that has been produced so far is a partial proof, which is provided below:

\begin{proof}
The inequality $\cat{f} \geq \max \{ k \mid f^\ast \beta_{\subsize{H}}^k \neq 0 \}$ is immediate by the cup length lower bound for $\cat{f}$. Thus, suppose that $f^\ast \beta_{\subsize{H}}^{n+1} = 0$. Using the Ganea-Schwarz approach to prove $\cat{f} \leq n$, it is sufficient to find a lift $\lambda$ for the diagram below:
\[\begin{tikzcd}[row sep = 3em]
&G_{n}(BH) \arrow[d, "p_{n}^{BH}"]
\\
BG \arrow[ur, dashed, "\lambda"] \arrow[r, swap, "Bf"]
&BH
\end{tikzcd}\]
The fiber $\ast^{n+1} \Omega BH$ of $p_n^{BH}$ is $(n-1)$-connected so that $\lambda$ can immediately be extended to the $n$-skeleton of $BG$. Now let $\theta^{n+1} \in H^{n+1}\left( G; \pi_{n} \left( \ast^{n+1}\Omega BH \right) \right)$ be the primary obstruction to extending the lift $\lambda$ to the $(n+1)$-skeleton of $BG$, and let $\phi^{n+1} \in H^{n+1}\left( H; \pi_{n} \left( \ast^{n+1}\Omega BH \right) \right)$ be the primary obstruction to extending a section of $p_n^{BH}$ to the $(n+1)$-skeleton of $BH$. Then
\[ \theta^{n+1} = f^\ast \phi^{n+1} = 0\]
by Lemma \ref{Berstein-Schwarz Lemma} and the naturality of the obstruction cocycle. Therefore, we can extend the lift $\lambda$ to the $(n+1)$-skeleton of $BG$.
\end{proof}

It may be possible to finish this partial proof using higher order obstructions, but for now it gives us a full proof for the following two special cases:

\begin{Corollary}
Let $f:G \to H$ be homomorphism between discrete groups. If $f^\ast \beta_{\subsize{H}}^n \neq 0$ and $f^\ast \beta_{\subsize{H}}^{n+1} = 0$ with $\dim BG \leq n+1$, then $\cat{f} = n$.
\end{Corollary}

\begin{Corollary}
Let $f:G \to H$ be a homomorphism between discrete groups such that $G$ is free. Then $\cat{f} = 1$ if and only if $f^\ast \beta_{\subsize{H}} \neq 0$, i.e., Conjecture \ref{LS Conjecture} holds.
\end{Corollary}

Additionally, we can show the case of free abelian groups:

\begin{Theorem} \label{free abelian category}
Let $f:\mathbb{Z}^n \to \mathbb{Z}^m$ be a group homomorphism. Then
\[ \cat{f} = \max \{ k \mid f^\ast \beta_{\subsize{\mathbb{Z}^m}}^k \neq 0\} = \text{\rm rank}(f).\]
\end{Theorem}

\begin{proof}
As before the inequality $\cat{f} \geq \max \{ k \mid f^\ast \beta_{\subsize{\mathbb{Z}^m}}^k \neq 0\}$ is immediate by the cup length lower bound for category. Let $j \leq m$ be the rank of $\text{im}(f)$ and note that $\mathbb{Z}^n = \ker (f) \oplus \text{im} (f)$; hence, we can break up $f$ into parts $f = f' \oplus g \oplus h$ where $f':\mathbb{Z}^j \to \mathbb{Z}^j$, $g:\mathbb{Z}^{n-j} \to 0$, and $h:0 \to \mathbb{Z}^{m-j}$ where $f'$ is injective. For all three of these maps we can easily calculate their LS-categories: $\cat{f'} = j$, $\cat{g} = 0$, and $\cat{h} = 0$. Therefore,
\[ j
= \cat{f'}
\leq \cat{f}
\leq \cat{f'} + \cat{g} + \cat{h}
= j \]
so that $\cat{f} = j = \text{rank}(f)$.

Now suppose that $f^\ast \beta^{k+1}_{\subsize{\mathbb{Z}^m}} = 0$. To finish the proof, it suffices to show that $j \leq k$. Note that we can choose generators $b_1,...,b_n$ of $\mathbb{Z}^n$ and $c_1,...,c_m$ of $\mathbb{Z}^m$ and integers $0 \neq \alpha_1,...,\alpha_j \in \mathbb{Z}$ such that $f(b_i) = \alpha_i c_i$ for $i \leq j$ and $f(b_i) = 0$ for $i > j$. Since $\mathbb{Z}^n$ and $\mathbb{Z}^m$ are abelian, it follows that $Bf_\ast = f$ on $H_1$. Thus, we can choose the dual elements $\hat{c}_1,...,\hat{c}_j \in H^1(\mathbb{Z}^m)$ of $c_1,...,c_j \in H_1(\mathbb{Z}^m)$. Next note that $f^\ast(\hat{c}_1 \cup ... \hat{c}_j) \neq 0$ because
\[ f_\ast \left( \left( b_1 \otimes ... \otimes b_j \right) \cap f^\ast \left( \hat{c}_1 \cup ...\cup \hat{c}_j \right) \right)
= f_\ast \left( b_1 \otimes ... \otimes b_j \right) \cap \left( \hat{c}_1 \cup ... \cup \hat{c}_j \right)\]
\[= \left( \alpha_1 c_1 \otimes ... \otimes \alpha_j c_j \right) \cap \left( \hat{c}_1 \cup ... \cup \hat{c}_j \right)
= \alpha_1 ... \alpha_j
\neq 0.\]
Then $j \leq k$ by Lemma \ref{Berstein-Schwarz Lemma} since $0 \neq f^\ast(\hat{c}_1 \cup ... \cup \hat{c}_j) \in H^j( \mathbb{Z}^n)$ and $f^\ast \beta_{\subsize{\mathbb{Z}^m}}^{k+1} = 0$.
\end{proof}

\begin{Lemma} \label{category homomorphism covering}
Let $f:G \to H$ be a group homomorphism. Then $\cat{f} = \secat{(Bf)^\ast u_{\subsize{H}}}$, where $u_{\subsize{H}}:EH \to BH$ is the universal cover of $BH$.
\end{Lemma}

\begin{proof}
Consider the following pullback diagrams:
\[\begin{tikzcd}
(Bf)^\ast EH \arrow[r] \arrow[d, swap, "(Bf)^\ast u_{\scaleto{H}{3.5pt}}"]
& EH \arrow[d, "u_{\scaleto{H}{3.5pt}}"]
&(Bf)^\ast P_0BH \arrow[r] \arrow[d, swap, "(Bf)^\ast p_0^{BH}"]
& P_0BH \arrow[d, "p_0^{BH}"]
\\
BG \arrow[r, swap, "Bf"]
& BH
& BG \arrow[r, swap, "Bf"]
& BH
\end{tikzcd}\]
Note that $\cat{Bf} = \secat{(Bf)^\ast p_0^{BH}}$ and our goal is to prove that $\cat{Bf}  = \secat{(Bf)^\ast u_{\subsize{H}}}$, so by Proposition \ref{genus invariance} it's sufficient to show that $(Bf)^\ast p_0^{BH}$ has a homotopy lift with respect to $(Bf)^\ast u_{\subsize{H}}$ and vice-versa. Then by the universal property of pullbacks, it suffices to prove that $p_0^{BH}$ has a homotopy lift with respect to $u_{\subsize{H}}$ and vice-versa. Since $P_0BH$ is contractible, $p_0^{BH}$ trivially satisfies the lifting criterion for the covering map $u_{\subsize{H}}$ so that there is some lift $L$ of $p_0^{BH}$ with respect to $u_{\subsize{H}}$. Since $P_0BH$ and $EH$ are both contractible, $L$ must be a homotopy equivalence with some homotopy inverse $\overline{L}$. Then
\[ p_0^{BH} \overline{L}
= \left( u_{\subsize{H}} L \right) \overline{L}
\simeq u_{\subsize{H}} Id_{EH}
= u_{\subsize{H}} \]
so that $\overline{L}$ is a homotopy lift of $u_{\subsize{H}}$ with respect to $p_0^{BH}$. Therefore, $\cat{Bf}  = \secat{(Bf)^\ast u_{\subsize{H}}}$.
\end{proof}

Now with this Lemma, if we assume $f$ is injective, then we get the following:

\begin{Theorem} Let $f:H \to G$ be an injective group homomorphism with $\cd{H} < \infty$. Then $\cat{f} = \cat{H}$ and so Conjecture \ref{LS Conjecture} holds.
\end{Theorem}

\begin{proof}
We immediately have $\cat{f} \leq \cat{H}$, so it's sufficient to prove $\cat{H} \leq \cat{f}$. By Lemma \ref{category homomorphism covering}, we know that $\cat{f} = \secat{(Bf)^\ast u_{\subsize{G}}}$. By applying Proposition \ref{genus invariance}, we can reduce proving $\cat{H} \leq \cat{f}$ to proving that there is a homotopy lift of $(Bf)^\ast u_{\subsize{G}}$ with respect to the covering map $u_{\subsize{H}}:EH \to BH$, i.e., we need to prove the dashed arrow exists in the following diagram:
\[\begin{tikzcd} [row sep = 3em, column sep = 3em]
EH \arrow[dr, swap, "u_{\scaleto{H}{3.5pt}}"] \arrow[rr, bend left, "Ef"]
& (Bf)^\ast EG \arrow[r, hook, swap, "\overline{Bf}"] \arrow[d, "(Bf)^\ast u_{\scaleto{G}{3.5pt}}"] \arrow[l, dashed]
& EG \arrow[d, "u_{\scaleto{G}{3.5pt}}"]
\\
& BH \arrow[r, swap, hook, "Bf"]
& BG
\end{tikzcd}\]
By the construction of $K(\pi,1)$-spaces, it follows that $Bf$ can be taken to be injective since $f$ is injective. Thus, the map $\overline{Bf}:(Bf)^\ast EG \to EG$ is injective. Since $EH$ and $EG$ are both contractible, it follows that $Ef$ is a homotopy equivalence with some homotopy inverse $L:EG \to EH$. Therefore, restricting $L$ to $(Bf)^\ast EG$, i.e., the map $L \overline{Bf}$, gives the homotopy lift we're looking for. Thus, $\cat{H} \leq \cat{f}$.
\end{proof}

\subsection{TC}

When the domain is an abelian group, the topological complexity reduces to the LS-category:

\begin{Theorem} \label{abelian group TC}
Let $f:G \to H$ be a homomorphism with $G$ abelian. Then $\cat{f} = \pullTC{f}$.
\end{Theorem}

\begin{proof}
Note that $BG$ is an H-space since $G$ is abelian. Therefore $\cat{f} = \pullTC{f}$ by Theorem \ref{H-space}.
\end{proof}

\begin{Corollary}
Let $f:\mathbb{Z}^n \to \mathbb{Z}^m$ be a homomorphism. Then $\cat{f} = \text{\rm rank}(f)$.
\end{Corollary}

\begin{proof}
This is immediate by Theorems \ref{abelian group TC} and \ref{free abelian category}.
\end{proof}

In \cite{Farber_3} and \cite{Farber_1}, it was shown that the topological complexity of a finite connected graph is the minimum of its first betti number and $2$. Since graphs are aspherical spaces with free fundamental group, it follows that $\TC{F(n)} = \min \{ 2,n \}$ where $F(n)$ is the free group on $n$-generators. For homomorphisms between free groups, the picture is a little bit more complicated but still achievable, but oddly enough the result still depends on the number of generators of the image.

\begin{Theorem} \label{Free Homomorphism TC}
Let $f:F(n) \to F(m)$ be a homomorphism between free groups. Then
\[\pullTC{f} = 
\begin{cases}
0 & \text{if } f = 0 \\
1 & \text{if } \text{\rm im} (f) \simeq \mathbb{Z} \\
2 & \text{otherwise}
\end{cases}\]
\end{Theorem}

\begin{proof}
It's immediate that $\pullTC{f} = 0$ if and only if $f = 0$, so it suffices to show that $\pullTC{f} = 1$ if and only if $\text{im} (f) \simeq \mathbb{Z}$.

First suppose that $\text{im} (f) \simeq \mathbb{Z}$. Then $f$ is not nullhomotopic so that $\pullTC{f} \geq 1$. Note that we can factor $f$ into $if'$ where $i:\text{im}(f) \to F(m)$ is the inclusion map and $f':F(n) \to \text{im}(f)$ is the map restricting the codomain of $f$ to $\text{im}(f)$. Then it follows that
\[ \pullTC{f} \leq \pullTC{f'} \leq \TC{\mathbb{Z}} = \TC{S^1}  = 1\]
so that $\pullTC{f} = 1$.

Now suppose that $\text{im}(f)$ has at least two distinct generators, say $h_1$ and $h_2$, and let $g_1$ and $g_2$ be such that $f(g_i) = h_i$. Then $\left< g_1,g_2 \right>$ is a free group on two generators and
\[ \pullTC{f|_{\left< g_1,g_2 \right>}} \leq \pullTC{f} \leq \TC{F(n)} \leq 2 \]
so that it suffices to prove $\pullTC{f} = 2$ when $n=2$ so that $f$ is injective. Consider the cohomology class
\[ \nu_m \in H^1 \left( F(m) \times F(m);I(F(m)) \right) \]
given by the representative map $(g,h) \mapsto gh^{-1} - 1$. Since $f$ is injective, it follows that $m \geq 2$ so that $\nu_m^2 \neq 0$ by \cite{Costa}. Note that $\nu_m \in \ker \Delta_H^\ast$ so that it suffices to prove that $(f \times f)^\ast \nu_m^2 \neq 0$ by Theorem \ref{pullback cohomology bound}. Note that we define $\nu_2$ in the same way as $\nu_m$ and similarly $\nu_2^2 \neq 0$. Since $f$ is injective, it induces an inclusion map $I(f):I(F(2)) \to I(F(m))$, and on cohomology $f \times f$ simply induces a restriction on the representatives
\[ [h:F(m) \times F(m) \to I(F(m))] \mapsto [h|_{\text{im}(f \times f)}].\]
Thus, it's easy to compute that
\[ (f \times f)^\ast \nu_m = I(f) \nu_2 \]
so that
\[ (f \times f)^\ast \nu_m^2 = \left( I(f) \otimes I(f) \right) \nu_2^2.\]
Recall that $I(F(2))$ is a free $\mathbb{Z}$-module so that, in particular, it is a flat $\mathbb{Z}$-module. Therefore, $I(f) \otimes I(f)$ is injective since $I(f)$ is injective. Thus, $(f \times f)^\ast \nu_m^2 \neq 0$ since $\nu^2_2 \neq 0$, so that $\pullTC{f} = 2$.
\end{proof}

\begin{Remark}
For $1$-dimensional CW-complexes, Theorem \ref{Free Homomorphism TC} implies that the topological complexity of a map can be characterized by what spaces the map can factor through up to homotopy, i.e., a point when $\pullTC{f} = 0$ and $S^1$ when $\pullTC{f} \leq 1$.
\end{Remark}

\footnotesize
\bibliographystyle{plain}
\bibliography{mapTC}

\begin{thebibliography}{10}

\bibitem{Garcia-Calcines_Garcia-Calcines_Vandembroucq}
J.G. Carrasquel-Vera, J.M. Garc\'{i}a-Calcines, and L.~Vandembroucq.
\newblock Relative category and monoidal topological complexity.
\newblock {\em Topology and its Applications}, 171:41--53, 2014.

\bibitem{Costa}
A.~Costa.
\newblock {\em Topological {C}omplexity of {C}onfiguration {S}paces}.
\newblock PhD thesis, Durham University, 2010.

\bibitem{Dranishnikov_1}
A.~Dranishnikov.
\newblock Topological {C}omplexity of {W}edges and {C}overing {M}aps.
\newblock {\em Proceedings of the American Mathematical Society},
  142:4365--4376, 2014.

\bibitem{Dranishnikov_2}
A.~Dranishnikov.
\newblock On {T}opological {C}omplexity of {H}yperbolic {G}roups.
\newblock {\em Proceedings of the American Mathematical Society},
  148:4547--4556, 2020.

\bibitem{Dranishnikov_Rudyak}
A.~Dranishnikov and Y.~Rudyak.
\newblock On the {B}erstein-{S}varc {T}heorem in {D}imension 2.
\newblock {\em Mathematical Proceedings of the Cambridge Philosophical
  Society}, 146:407--413, 2009.

\bibitem{Eilenberg_Ganea}
S.~Eilenberg and T.~Ganea.
\newblock On the {L}usternik-{S}chnirelmann {C}ategory of {A}bstract {G}roups.
\newblock {\em Annals of Mathematics}, 65:517--518, 1957.

\bibitem{Farber_2}
M.~Farber.
\newblock Topological {C}omplexity of {M}otion {P}lanning.
\newblock {\em Discrete Comput Geom}, 29:211--221, 2003.

\bibitem{Farber_3}
M.~Farber.
\newblock Instabilities of {R}obot {M}otion.
\newblock {\em Topology and its Applications}, 140:245--266, 2004.

\bibitem{Farber_1}
M.~Farber.
\newblock Topology of {R}obot {M}otion {P}lanning.
\newblock In {\em Morse Theoretic Methods in Nonlinear Analysis and in
  Symplectic Topology}, pages 185--230. 2006.

\bibitem{Farber_4}
M.~Farber.
\newblock {\em Invitation to {T}opological {R}obotics}.
\newblock Z\"urich Lectures in Advanced Mathematics. European Mathematical
  Society (EMS), Z\"urich, 2008.

\bibitem{Farber_Mescher}
M.~Farber and Stephan Mescher.
\newblock On the {T}opological {C}omplexity of {A}spherical {S}paces.
\newblock {\em Journal of Topology and Analysis}, 12:293--319, 2020.

\bibitem{Garcia-Calcines}
J.M. Garc\'{i}a-Calcines.
\newblock A {N}ote on {C}overs {D}efining {R}elative and {S}ectional
  {C}ategories.
\newblock {\em Topology and its Applications}, 165:106810, 2019.

\bibitem{Iwase_Sakai}
N.~Iwase and M.~Sakai.
\newblock Erratum to ``{T}opological {C}omplexity is a {F}iberwise {L}-{S}
  {C}ategory" [{T}opology {A}ppl. 157 (1) (2010) 10–21].
\newblock {\em Topology and its Applications}, 159:2810--2813, 2012.

\bibitem{Miyata}
T.~Miyata.
\newblock Fibrations in the {C}ategory of {A}bsolute {N}eighborhood {R}etracts.
\newblock {\em Bulletin of the Polish Academy of Sciences Mathematics},
  55:145--154, 2007.

\bibitem{Murillo_Wu}
A.~Murillo and J.~Wu.
\newblock Topological {C}omplexity of the {W}ork {M}ap.
\newblock {\em Journal of Topology and Analysis}, 12, 2020.

\bibitem{Pavesic_1}
P.~Pave\v{s}i\'{c}.
\newblock Complexity of the {F}orward {K}inemtaic {M}ap.
\newblock {\em Mechanism and Machine Theory}, 117:230--243, 2017.

\bibitem{Pavesic_2}
P.~Pave\v{s}i\'{c}.
\newblock A {T}opologist's {V}iew of {K}inematic {M}aps and {M}anipulation
  {C}omplexity.
\newblock {\em Contemporary Mathematics}, 702:61--83, 2018.

\bibitem{Pavesic_3}
P.~Pave\v{s}i\'{c}.
\newblock Topological {C}omplexity of a {M}ap.
\newblock {\em Homology, Homotopy and Applications}, 21:107--130, 2019.

\bibitem{Schwarz}
A.~S. Schwarz.
\newblock The {G}enus of a {F}iber {S}pace.
\newblock In {\em American Mathematical Society Translations Series 2 Volume
  55}, pages 49--140. 1966.

\bibitem{Short}
R.~Short.
\newblock Relative {T}opological {C}omplexity of a {P}air.
\newblock {\em Topology and its Applications}, 248:7--23, 2018.

\bibitem{Srinivasan}
T.~Srinivasan.
\newblock On the {L}usternik-{S}chnirelmann {C}ategory of {P}eano {C}ontinua.
\newblock {\em Topology and its Applications}, 160:1742--1749, 2013.

\bibitem{Stanley}
D.~Stanley.
\newblock On the {L}usternik-{S}chnirelman {C}ategory of {M}aps.
\newblock {\em Canad. J. Math.}, 54 (3):608--633, 2002.

\end{thebibliography}

\end{document}